\documentclass[letterpaper, 10 pt, conference]{ieeeconf}  

\IEEEoverridecommandlockouts                              

\usepackage{xcolor}
\usepackage{amsmath}
\usepackage{amssymb}
\usepackage{mathtools}
\usepackage{multirow}
\newtheorem{lemma}{Lemma}
\newtheorem{rem}{Remark}
\newtheorem{theorem}{Theorem}
\newtheorem{prop}{Proposition}
\newtheorem{assumption}{Assumption}

\usepackage{graphicx}
\usepackage{hyperref}

\title{\LARGE \bf Control of Uncertain PWA Systems using Difference-of-Convex Decompositions
}

\author{Siddharth H. Nair, Yvonne R. St\"urz 
\thanks{Email IDs:\{siddharth$\_$nair, y.stuerz\}@berkeley.edu}
}

\begin{document}

\maketitle
\thispagestyle{empty}
\pagestyle{empty}

\begin{abstract}
In this work, we analyse and design feedback policies for discrete-time Piecewise-Affine (PWA) systems with uncertainty in both the affine dynamics and the polytopic partition. The main idea is to utilise the Difference-of-Convex (DC) decomposition of continuous PWA systems to derive quadratic Lyapunov functions as stability certificates and stabilizing affine policies in a higher dimensional space. When projected back to the state space, we obtain time-varying PWQ Lyapunov functions and time-varying PWA feedback policies. 

\end{abstract}

\section{Introduction}
A discrete-time Piecewise-Affine (PWA) system is described by the dynamics
\begin{align}\label{eq:dyn}
    x(t+1)&=f(x(t),u(t))\\
          &=A_ix(t)+B_iu(t)+c_i,~~\forall(x(t),u(t))\in\Omega_i\nonumber
\end{align}
where $x(t), u(t)$ are the state and input respectively at time $t$, and $\Omega_i$, $i=1,\dots, n_r$ form a polytopic partition of the state and input space $\Omega=\mathcal{X}\times\mathcal{U}\subset\mathbb{R}^{n+m}$, denoted as $\cup\Omega_i=\Omega$. PWA systems are a versatile modeling framework; They can approximate smooth nonlinear dynamics arbitrarily well \cite{sontag1981nonlinear} and have also been shown to be equivalent to certain classes of hybrid systems such as Mixed Logical Dynamical (MLD) systems and Linear Complementarity (LC) systems under mild conditions \cite{heemels2001equivalence}. More recently in \cite{hempel2013every}, it was shown that continuous PWA systems can be written as the difference of two continuous convex PWA functions and that they can be written as optimizing processes.

The analysis and control of PWA systems has been widely studied in the literature. In \cite{ferrari2002analysis, mignone2000stability, hovd2013relaxing}, stabilizing feedback controllers are synthesized using quadratic and Piecewise-Quadratic (PWQ) Lyapunov functions as stability certificates. PWQ Lyapunov functions, being a generalisation of quadratic Lyapunov functions, provide less conservative stability certificates but are more difficult to synthesize. They require an explicit description of the polytopic partition $\sqcup\Omega_i=\Omega$ and need to account for all possible transitions between the polytopic pieces $\Omega_i$ explicitly. Similarly, the works \cite{rubagotti2011stability, prajna2003analysis} consider piecewise-linear and piecewise-polynomial Lyapunov functions for stability.  A significantly more challenging problem is the synthesis of optimal feedback control of PWA systems, studied in \cite{bemporad2000optimal, borrelli2005dynamic, kerrigan2002optimal}. Here, the control problem is complicated by the fact that the stabilizing controller must transition between polytopic parititions in an optimal fashion over  a finite-time horizon. In this article, we focus on the simpler task of synthesizing stabilizing feedback control but for uncertain PWA systems.

A PWA system is said to be uncertain if the dynamics function $f(\cdot)$ is not known exactly. Uncertain PWA systems are relatively less studied in the literature due to the complexity of uncertain switching systems. The uncertainty descriptions considered in the literature assume some form of parametric uncertainty in the matrices $A_i$, $B_i$, $c_i$ but with known, static partition $\cup\Omega_i=\Omega$ \cite{gao2009robust,trimboli2011stability,hovd2018parameter}. This uncertainty description is justified when the PWA framework is used to approximate a smooth nonlinear system. However for a system that is inherently hybrid and uncertain, the polytopic partition $\cup\Omega_i(t)=\Omega$ itself may be dynamic with time-varying polytopic pieces $\Omega_i(t)$.

In this article, we investigate the control of uncertain PWA systems where both the affine dynamics and the polytopic partition are uncertain leveraging results in the control and analysis of nominal PWA systems from \cite{hempel2014inverse, groff2019stability, aydinoglu2019contact}. The main challenge in the synthesis of stabilizing control of such systems is the derivation of a stability certificate that does not require an explicit description of the uncertain, time-varying polytopic partition $\cup\Omega_i(t)=\Omega$. An obvious solution to this problem would be to construct a common quadratic Lyapunov function on the state space but this has been shown to be overly conservative \cite{ferrari2002analysis}. We show how this issue can be addressed for continuous PWA systems for which a Difference of Convex (DC) functions decomposition can be derived \cite{hempel2013every}. Inspired by graph approximation techniques in \cite{groff2019stability, aydinoglu2019contact, nair2020modeling}, we show that the DC decomposition permits a construction of affine stabilizing policies with quadratic Lyapunov functions in a higher dimensional space, which when projected back to the state space result in time-varying PWA policies and time-varying PWQ Lyapunov functions.

The remainder of the article is organized as follows. We briefly discuss our notations in Section~\ref{sec:notations} before formally stating our problem statement in Section~\ref{sec:PD}. Section~\ref{sec:lift} presents the key technique we use for expressing the PWA dynamics in a higher dimensional space. Section~\ref{sec:CLF_pol} describes our synthesis procedure in detail which we demonstrate on two applications in Section~\ref{sec:numerics}.

\section{Notation} \label{sec:notations}

$\Vert\cdot\Vert$ denotes the Euclidean norm in $\mathbb{R}^n$ unless explicitly stated otherwise. With slight abuse of notation, $\Vert\cdot\Vert$ also denotes the norm on $\mathbb{R}^{n\times n}$ induced by the Euclidean norm. $I_n\in \mathbb{R}^{n\times n}$ is the identity matrix  and $\mathbf{0}$ is a matrix of zeros with relevant dimensions. $[x]_i$ denotes the $i$th element of vector $x\in\mathbb{R}^n$, $[M]_{ij}$ denotes the $i$th element in the $j$th column of matrix $M\in\mathbb{R}^{n\times m}$ and $[M]_i$ denotes its $i$th column. We denote the Hermitian of a real square matrix $M$ as $\text{He}(M)=\frac{M+M^\top}{2}$.  Given matrices $A_1,\dots, A_N\in\mathbb{R}^{n\times m}$, we denote their convex hull as the set $\text{conv}\{A_k\}_{k=1}^N\subset\mathbb{R}^{n\times m}$.  Given matrices $A_1,\dots A_N\in\mathbb{R}^{n\times n}$, we denote by $\text{blkdiag}(A_1,\dots,A_N)\in\mathbb{R}^{Nn\times Nn}$ as the matrix with block diagonals given by $A_1,\dots, A_N$ and zeros everywhere else. As a shorthand, we write the symmetric block matrix $\begin{bmatrix}A &B\\B^\top&C\end{bmatrix}$ as $\begin{bmatrix}A&B\\\star&C\end{bmatrix}$.
\section{Problem Formulation}\label{sec:PD}
We start with the Difference of Convex (DC) decomposition of $f(\cdot)$ in \eqref{eq:dyn}. As shown in \cite{hempel2013every}, all continuous PWA functions defined on convex polytopes admit a DC decomposition and so we restrict our focus to such systems.
\begin{assumption}\label{assmp:PWA}
The PWA system dynamics $f(\cdot)$ in \eqref{eq:dyn} is continuous  $\forall(x(t),u(t))\in\Omega$ where $\Omega=\mathcal{X}\times\mathcal{U}$ is a convex set with polyhedral partition $\Omega=\cup\Omega_i$.
\end{assumption}
Consider the DC decomposition of the dynamics \eqref{eq:dyn} given by
\begin{align}\label{eq:DC}
    x(t+1)&=Ax(t)+\gamma(x(t))-\eta(x(t))+Bu(t),\\
    [\gamma(x)]_i&=\max_{1\leq j\leq \alpha}\ [E_j]_i x+[d_j]_i,\nonumber\\
    [\eta(x)]_i&=\max_{1\leq l\leq \beta}\ [H_l]_i x+[f_l]_i,~~ i=1,\dots, n \nonumber
\end{align}
where without loss of generality, we have assumed that the convex functions $\gamma(\cdot),\eta(\cdot)$ are functions of state $x(t)$ only and the number of affine functions $\alpha$ and $\beta$ for $\gamma(\cdot)$ and $\eta(\cdot)$ respectively are the same for each dimension $i=1,\dots,n$. This DC decomposition can be obtained using the constructive procedure outlined in Lemma 4 of \cite{hempel2013every}.
\begin{rem} If $\gamma(\cdot), \eta(\cdot)$ are functions of $u$, we augment the state space as $\tilde{x}=[x^\top u^\top]$ by adding integrators.
\end{rem}
\begin{rem}
If $[\gamma(x)]_i=\max_{1\leq j\leq \alpha_i}\ [E_j]_i x+[d_j]_i$, define $\alpha=\max_{1\leq i\leq n}\alpha_i$ and add $\alpha-\alpha_i$ affine functions that lower bound the convex function $[\gamma(x)]_i$ on the set $\mathcal{X}$ to get $[\gamma(x)]_i=\max_{1\leq j\leq \alpha}\ [E_j]_i x+[d_j]_i$. A similar procedure can be repeated for $\eta(x)$.
\end{rem}
Suppose that the DC system in \eqref{eq:DC} is uncertain, i.e., the matrices $E_j, d_j, H_j,f_j$ are uncertain and time-varying, but bounded. We formalize this in the assumption below.
\begin{assumption}\label{assmp:unc}
The matrix-valued random variable $\mathbf{W}(t)\in\mathbb{R}^{n\times (n+1)(\alpha+\beta)}$ given by \begin{align*}\mathbf{W}(t)=[E_1(t)\ d_1(t)\dots E_\alpha(t)\ d_\alpha(t)\ H_1(t)\ f_1(t)\\
\dots H_\beta(t)\ f_\beta(t)]
\end{align*}is independent, identically distributed (i.i.d) for all time $t\geq0$ and supported on a known compact convex polytope $\mathcal{W}=\text{conv}\{W^k\}_{k=1}^{n_V}\subset\mathbb{R}^{n\times (n+1)(\alpha+\beta)}$
where $W^k=[E^k_1\ d^k_1\dots E^k_\alpha\ d^k_\alpha\ H^k_1\ f^k_1\dots H^k_\beta\ f^k_\beta]$.
\end{assumption}
To emphasize that system \eqref{eq:DC} is time-varying, we write the uncertain convex functions as $\gamma(t,x(t)), \eta(t,x(t))$. This system description captures the scenario of PWA systems with uncertainty in both the PWA function definitions ($A_i,B_i,c_i$ for \eqref{eq:dyn}) and the polytopic partitions ($\Omega_i$ for \eqref{eq:dyn}).

We also assume that the origin $x=\mathbf{0}$ is an unforced ($u=\mathbf{0}$) equilibrium for \eqref{eq:DC}. More formally, we assume the following \textit{matching} condition.
\begin{assumption}\label{assmp:eqlbrm}
The origin is an unforced equilibrium of system \eqref{eq:DC}, i.e.,
$$\gamma(t,\mathbf{0})=\eta(t,\mathbf{0})\ \forall\mathbf{W}(t)\in\mathcal{W},  t\geq0\Leftrightarrow\max_{1\leq j\leq \alpha} [d_j^k]_i=\max_{1\leq l\leq \beta}[f_l^k]_i $$
for all $i=1,\dots, n$ and $k=1,\dots, n_V$.
\end{assumption}

In this article, we aim to synthesize PWA feedback policies \begin{align}\label{eq:tvPWA_pol}
    u(t,x)=F_i(t)x+g_i(t)\ \forall x\in\tilde{\mathcal{X}}_i(t)
\end{align} with disjoint partition $\tilde{\mathcal{X}}_i(t)$, s.t.  $\cup\tilde{\mathcal{X}}_i(t)=\mathcal{X}$,  such that the origin is exponentially stable for the closed-loop system. Notice that the partition on which the policy is defined is \textit{time-varying}.   This class of uncertain PWA systems has not been studied in the literature to the best of our knowledge. Most approaches for feedback control synthesis for PWA systems in the literature rely on an explicit description of the \text{static} polytopic partitions and construct affine policies local to each partition. This is impossible for the PWA systems considered in this work because the polytopic partitions are dynamic (time-varying) and unknown. Therefore, to deal with these uncertain hybrid systems, we present a technique that allows us to design PWA policies without explicitly computing the polytopic partitions. We illustrate our approach on two uncertain PWA systems in Section~\ref{sec:numerics}: A) an inverted pendulum against a soft wall and B) a robot collaborating with a human to transport a payload.
\section{Lifted Quadratic Representations for Convex Piecewise Affine Functions}\label{sec:lift}
In this section, we describe the key technique for verifying if a continuous PWA function satisfies a given quadratic inequality without explicitly computing the polytopic partitions. First, consider the element-wise convex PWA function $\lambda: \mathbb{R}^n\rightarrow \mathbb{R}^n$, $$\lambda(x)=\max_{1\leq j\leq \xi} \{E_{j}x+d_j\}$$ where the $\max\{\cdot\}$ operation is defined element-wise. Recursively define the following functions 

\begin{align}\label{eq:rec_lmbd_def}
    \lambda_j(x)=\max\{\lambda_{j-1}(x), E_{j}x+d_j\}\quad j=2,\dots,\xi
\end{align}

with $\lambda_1(x)=E_{1}x+d_1$. Note that due to the fact that the $\max\{\cdot\}$ operation can be equivalently written in a nested fashion (for e.g., $\max\{p,q,r\}=\max\{\max\{p,q\},r\}$), we actually have $\lambda_\xi(x)=\lambda(x)$. 

Each function $\lambda_{j}(x)$ defined in \eqref{eq:rec_lmbd_def} can be equivalently written as the unique solution to $2n$ affine inequalities and $1$ quadratic equality in $(x,\lambda_{j-1}(x))$ as shown below. 
\begin{align}\label{eq:max_solnset}
&\lambda_1(x)=E_1x+d_1,\nonumber\\
    &\lambda_{j}(x)\geq \lambda_{j-1}(x),\ \lambda_{j}(x)\geq E_{j}x+d_j,\\
    &(\lambda_{j}(x)-\lambda_{j-1}(x))^\top(\lambda_j(x)-E_{j}x-d_j)=0,\ j=2,\dots, \xi.\nonumber
\end{align}
Given any $x\in\mathbb{R}^n$, the functions $\lambda_j(x)
~\forall j=1,\dots,\xi$ are the unique solutions to \eqref{eq:max_solnset}.
Consequently, the function $\lambda(x)$ (or equivalently, its graph $(x,\lambda(x))$) is given by the solution to a set of affine inequalities and quadratic equalities given in \eqref{eq:max_solnset} in the space $(x,\lambda_1, \dots,\lambda_\alpha)\in\mathbb{R}^{n(\xi+1)}$ where $x,\lambda_1,\dots,\lambda_\alpha$ are arbitrary variables. The next proposition offers a useful tool to construct a convex relaxation of the set of solutions $(x,\lambda(x))$ given by \eqref{eq:max_solnset}.  

\begin{prop}\label{prop:lmbd_graph}
Let $$\chi(x)=[1\ x^\top\ \lambda_1(x)^\top\ \lambda_2(x)^\top\dots \lambda_{\xi}(x)^\top]^\top$$ and define the selection matrices $S_j\in\mathbb{R}^{(1+3n)\times (n(\xi+1)+1)}$ such that $S_j\chi(x)=[1\ x^\top\ \lambda_{j-1}(x)^\top\ \lambda_{j}(x)^\top]^\top$ for $j=2,\dots, \xi$ and $S_1\in\mathbb{R}^{(1+2n)\times(n(\xi+1)+1)}$ such that $S_1\chi(x)=[1\ x^\top\ \lambda_1(x)^\top]^\top$. Given any diagonal matrices $D_j\in\mathbb{R}^{n\times n}$ and any non-negative matrices $M_j\in\mathbb{R^+}^{n\times n}$ $~\forall j=1,\dots, \xi$, define the set 
\begin{align}\label{eq:lmb_grph}
    \mathfrak{X}(x)=\{\chi&=[1\ x^\top\ \lambda_1^\top\dots\lambda_\xi^\top]^\top \in\mathbb{R}^{n(\xi+1)+1}:\nonumber\\ &\chi^\top S_j^\top G_j^\top D_j \bar{G}_j S_j\chi=0,~\chi^\top S_j^\top G_j^\top M_j \bar{G}_jS_j\chi\geq 0\nonumber\\
    &\forall j=1,\dots,\xi.\}
\end{align}
where $G_j=[0\ \mathbf{0}\ -I_n\ I_n]$, $\bar{G}_j=[-d_j\ -E_{j}\ \mathbf{0}\ I_n]$ for $j=2,\dots\xi$ and $G_1=[1\  \mathbf{0}\ \mathbf{0}]$, $\bar{G}_1=[-d_1\ -E_1\ I_n]$. Then 
$$\chi(x)\in\mathfrak{X}(x).$$
\end{prop}

\begin{proof}
First see that for any $j=2,\dots,\xi$, we have $G_jS_j\chi(x)=\lambda_j(x)-\lambda_{j-1}(x)$ and $\bar{G}_jS_j\chi(x)=\lambda_j(x)-A_{j}x-d_j$. For the equality in \eqref{eq:max_solnset}, we have 
\begin{align*}
    &(\lambda_{j}(x)-\lambda_{j-1}(x))^\top(\lambda_j(x)-E_{j}x-d_j)=0\\
    \Rightarrow&([\lambda_{j}(x)-\lambda_{j-1}(x)]_i)[D_j]_{ii}([\lambda_j(x)-E_{j}x-d_j]_i)=0, \\
    &~~~~~~~~~~~~~~~~~~~~~~~~~~~~~~~~\forall i=1,\dots, n\\
    \Leftrightarrow& \chi^\top S_j^\top G^\top D_j \bar{G}_j S_j\chi=0
\end{align*}
where $D_j\in\mathbb{R}^{n\times n}$ is any diagonal matrix.
For the inequalities in \eqref{eq:max_solnset}, we have for any non-negative matrix $M_j\in\mathbb{R^+}^{n\times n}$ that
\begin{align*}
    &\lambda_{j}(x)\geq \lambda_{j-1}(x),\ \lambda_{j}(x)\geq E_{j}x+d_j\\
    \Rightarrow (&[\lambda_{j}(x)-\lambda_{j-1}(x)]_i)([\lambda_j(x)-E_{j}x-d_j]_k)\geq0,\\
    &~~~~~~~~~~~~~~~~~~~~~~~~~~~~~~~~\forall i,k=1,\dots, n\\
    \Rightarrow (&[\lambda_{j}(x)-\lambda_{j-1}(x)]_i)[M_j]_{ik}([\lambda_j(x)-E_{j}x-d_j]_k)\geq0,\\
    &~~~~~~~~~~~~~~~~~~~~~~~~~~~~~~~~\forall i,k=1,\dots, n\\
    \Leftrightarrow & \chi^\top S_j^\top G^\top M_j G_jS_j\chi\geq 0.
\end{align*}

For $j=1$, see that $G_1S_1\chi(x)=1$ and $\bar{G}_1S_1\chi(x)=\lambda_1(x)-E_{1}x-d_1$. Proceeding as above, we get the result for $j=1$ where the inequality $\chi(x)S_1^\top G_1^\top M_1\bar{G}_1S_1\chi(x)\geq 0$ holds trivially with equality.
\end{proof}
The above proposition provides us with a tool to construct a sufficiency certificate for a quadratic inequality $[1\ x^\top\ \lambda_1^\top\ \dots\lambda_\xi^\top]P[1\ x^\top\ \lambda_1^\top\ \dots\lambda_\xi^\top]^\top\geq 0$ to hold given that $[1\ x^\top \lambda_1^\top\dots\lambda_\xi^\top]^\top=\chi(x)$. Using the S-procedure, if $\exists D_j\in\mathbb{R}^{n\times n}, M_j\in\mathbb{R^+}^{n\times n}$ for $j=1,\dots,\xi$ such that $$P-\sum_{j=1}^\xi\text{He}(S^\top_jG_j^\top(M_j+D_j)\bar G_jS_j)\geq 0,$$
then $\chi(x)^\top P\chi(x)\geq 0$. This can be readily verified by multiplying both sides of the above inequality by $\chi(x)$ and using Proposition~\ref{prop:lmbd_graph}.

To use Proposition~\ref{prop:lmbd_graph} for a general continuous PWA function $\lambda:\mathbb{R}^n\rightarrow\mathbb{R}^n$, we first obtain its DC decomposition $\lambda(x)=\gamma(x)-\eta(x)$. Then we can use proposition \ref{prop:lmbd_graph} for $\gamma(\cdot)$ and $\eta(\cdot)$ to characterize the set $(x,\gamma(x), \eta(x))$ in the space $(x, \gamma_1,\dots,\gamma_\alpha, \eta_1,\dots, \eta_\beta)\in\mathbb{R}^{n(1+\alpha+\beta)}$. We illustrate this procedure by synthesizing a PWA policy and PWQ Lyapunov function for the uncertain DC system \eqref{eq:DC}.

\section{Synthesis of Stabilizing PWA Policies}\label{sec:CLF_pol}
In this section, we provide a procedure for synthesizing stabilizing PWA feedback policies for uncertain DC systems \eqref{eq:DC} using PWQ Control Lyapunov Functions (CLFs) as stability certificates. First, we provide conditions under which a function qualifies as a CLF which in turn, establishes the existence of a stabilizing controller for system \eqref{eq:DC}. Second, we propose a redefinition of the DC decomposition in model \eqref{eq:DC} without changing the system dynamics. Finally, we use the redefined model to synthesize a PWQ CLF and a PWA control policy using Linear Matrix Inequalties (LMIs).
\subsection{Stability via Control Lyapunov Functions}
\begin{lemma}\label{lem:CLF}
Suppose that there exists a function $V:\mathbb{R}^+\times\mathcal{X}\rightarrow\mathbb{R}$ and scalars $\rho_1,\rho_2\in\mathbb{R}^+, \rho_3\in(0,1)$ such that:
\begin{enumerate}
    \item $\rho_1\Vert x\Vert^2\leq V(t,x)\leq \rho_2\Vert x\Vert
    ^2\ \forall x\in\mathcal{X}, t\geq 0$
    \item $\forall t\geq0, x(t)\in\mathcal{X},\exists u^*(t,x(t))\in\mathcal{U} :$\\ $V(t+1,x(t+1))\leq\rho_3 V(t,x(t)),\ \forall \mathbf{W}(t)\in\mathcal{W}$.
\end{enumerate}
Then $V(t,x)$ is a Control Lyapunov function for system \eqref{eq:DC}, and the origin of the closed-loop system with $u(t)=u^*(t,x(t))$ is exponentially stable.
\end{lemma}
\begin{proof}
Let $V(t,x)$ be a CLF for system \eqref{eq:DC}. Then for any initial condition $x(0)=x_0\in\mathcal{X}$, we have $V(0,x(0))\leq \rho_2 \Vert x_0\Vert^2$ using the second inequality in property 1). Using property 2) for trajectories of the uncertain system \eqref{eq:DC} in closed-loop with $u^*(t,x(t))$, we also have 
$$ V(t,x(t))\leq \rho_3 V(t-1,x(t-1))\leq\dots\leq \rho_3^tV(0,x(0))$$
for all $\mathbf{W}(0),\dots,\mathbf{W}(t-1)\in\mathcal{W}$.
Using the first inequality in property 1), we get $\rho_1\Vert x(t)\Vert^2\leq V(t,x(t))$. Combining these inequalities, we get $\Vert x(t)\Vert\leq \sqrt{\rho_3^t\frac{\rho_2}{\rho_1}}\Vert x_0\Vert$ which in turn, implies $\lim_{t\rightarrow\infty}\Vert x(t)\Vert=0$ because $\rho_3\in(0,1)$ and $\frac{\rho_2}{\rho_1}>0$. Since $x_0$ was arbitrary, the proof is complete.
\end{proof}
In order to present CLF candidates for system \eqref{eq:DC}, we redefine its DC decomposition without effectively changing the dynamics in the next section.
\subsection{Redefinition of DC Decomposition}
\begin{lemma}\label{lem:DC_redef}
 Let $d_*^{k}\in\mathbb{R}^n$ such that $$[d_*^{k}]_i=\max_{1\leq j\leq \alpha} [d_j^k]_i=\max_{1\leq l\leq \beta}[f_j^k]_i$$ for system \eqref{eq:DC}, which exists by assumption \ref{assmp:eqlbrm}.  Let $\tilde{E}^k_j=E^k_j, \tilde{d}^k_j=d^k_j-d^k_*$, for $j=1,\dots,\alpha$ and $\tilde{H}^k_j=H^k_j, \tilde{f}^k_l=f^k_l-d^k_*$ , for $l=1,\dots,\beta$. Accordingly redefine the matrices $\tilde{\mathbf{W}}(t)$, $\tilde{W}^k$ and polytope $\tilde{\mathcal{W}}$ as in Assumption \ref{assmp:unc}. Then the redefined DC decomposition of system \eqref{eq:DC} given by 
 \begin{align}\label{eq:DC_redef}
   x(t+1)&=Ax(t)+\tilde{\gamma}(t,x(t))-\tilde{\eta}(t,x(t))+Bu(t),\\
     [\tilde{\gamma}(t,x)]_i&=\max_{1\leq j\leq \alpha}\ [\tilde{E}_j(t)]_i x+[\tilde{d}_j(t)]_i,\nonumber\\
    [\tilde{\eta}(t,x)]_i&=\max_{1\leq l\leq \beta}\ [\tilde{H}_l(t)]_i x+[\tilde{f}_l(t)]_i,~~ i=1,\dots, n\nonumber
 \end{align}
 satisfies the following properties
    \begin{enumerate}
        \item $\gamma(t,x)-\eta(t,x)=\tilde{\gamma}(t,x)-\tilde{\eta}(t,x), \forall t\geq 0,x\in\mathcal{X}$
        \item $\tilde{\gamma}(t,\mathbf{0})=\tilde{\eta}(t,\mathbf{0})=\mathbf{0}, \forall t\geq 0$
        \item $\exists C_\gamma, C_\eta\in\mathbb{R}^+\ : \Vert\tilde{\gamma}(t,x)\Vert\leq C_\gamma\Vert x\Vert, \Vert\tilde{\eta}(t,x)\Vert\leq C_\eta\Vert x\Vert$.
    \end{enumerate}
\end{lemma}
\begin{proof}
\begin{enumerate}
    \item We have $\mathbf{W}(t)\in\mathcal{W}$ from Assumption \ref{assmp:unc}, therefore \begin{align*}
    &\forall t\geq 0,\
\exists v(t)\in(\mathbb{R}^+)^{n_V}, \mathbf{1}^\top v(t)=1:\\  &\sum_{k=1}^{n_V}[v(t)]_kW^k=\mathbf{W}(t).
\end{align*} 
    Define $d_*(t)=\sum_{k=1}^{n_V}[v(t)]_kd^k_*$ and see that for any $t\geq 0, x\in\mathcal{X}$,
    \begin{align*}
     &[\gamma(t,x)-\eta(t,x)]_i~~~~~~~~~~~~~~~~~~~~~~~~~~~~~~~~~~~~~~~\\
     &~~=\max_{1\leq j\leq \alpha}\ [E_j(t)]_i x(t)+[d_j(t)]_i\\
     &~~~~-\max_{1\leq l\leq \beta}\ [H_l(t)]_i x(t)+[f_l(t)]_i\\
     &~~=-d_*(t)+\max_{1\leq j\leq \alpha}\ [E_j(t)]_i x(t)+[d_j(t)]_i\\
     &~~~~+d_*(t)-\max_{1\leq l\leq \beta}\ [H_l(t)]_ix(t)+[f_l(t)]_i\\
     &~~=\sum_{k=1}^{n_V}[v(t)]_k\big(\max_{1\leq j\leq \alpha}\ [E_j^k]_i x(t)+[d_j^k-d^k_*]_i\\
     &~~~~-\max_{1\leq l\leq \beta}\ [H_l^k]_i x(t)+[f_l^k-d^k_*]_i\big)\\
     &~~=\sum_{k=1}^{n_V}[v(t)]_k\big(\max_{1\leq j\leq \alpha}\ [\tilde{E}_j^k]_i x(t)+[\tilde{d}_j^k]_i\\
     &~~~~-\max_{1\leq l\leq \beta}\ [\tilde{H}_l^k]_i x(t)+[\tilde{f}_l^k]_i\big)\\
     &~~=[\tilde{\gamma}(t,x)-\tilde{\eta}(t,x)]_i,~~i=1,\dots,n.
    \end{align*}
    \item Let $\tilde{j}^i_k=\text{arg}\max_{1\leq j\leq \alpha}[d_j^k]_i$ and $\tilde{l}^i_k=\text{arg}\max_{1\leq l\leq \beta}[f_l^k]_i$. By Assumption \ref{assmp:eqlbrm}, we have $[d^k_{\tilde{j}^i_k}]_i=[f^k_{\tilde{l}^i_k}]_i=[d^k_*]_i$. Therefore, $\tilde{d}^k_j, \tilde{f}^k_j\leq 0$ and $\forall i=1,\dots, n$,
    \begin{align*}
&[\tilde{\gamma}(t,\mathbf{0})]_i=\max_{1\leq j\leq\alpha}[\tilde{d}^j_k]_i=[d^k_{\tilde{j}^i_k}]_i-[d^k_*]_i= 0,\\ &[\tilde{\eta}(t,\mathbf{0})]_i=\max_{1\leq l\leq\beta}[\tilde{f}^j_k]_i=[f^k_{\tilde{l}^i_k}]_i-[d^k_*]_i= 0.
\end{align*}
\item Let $C_\gamma=\max_{\substack{1\leq j\leq\alpha\\1\leq k\leq n_V}}\Vert E^k_j\Vert_F$ (where $\Vert\cdot\Vert_F$ is the Frobenius norm). We have already seen in 2) that $\tilde{d}^k_j\leq 0$ and so we have 
\begin{align*}
    &E^k_jx+\tilde{d}^k_j\leq E^k_jx, \forall j=1,\dots,\alpha\\
\Rightarrow&\sum_{k=1}^{n_V}[v(t)]_k(\max_{1\leq j\leq \alpha}E^k_jx+\tilde{d}^k_j)\leq\sum_{k=1}^{n_V}[v(t)]_k(\max_{1\leq j\leq \alpha}E^k_jx)\\
\Rightarrow&\tilde{\gamma}(t,x)\leq \sum_{k=1}^{n_V}[v(t)]_k(\max_{1\leq j\leq \alpha}E^k_jx)\\
\Rightarrow&\Vert\tilde{\gamma}(t,x)\Vert\leq(\sum_{k=1}^{n_V}[v(t)]_k) C_\gamma\Vert x\Vert=C_\gamma\Vert x\Vert
\end{align*}
By setting $C_\eta=\max_{\substack{1\leq l\leq\beta\\1\leq k\leq n_V}}\Vert H^k_l\Vert_F$, we can similarly show $\Vert\tilde{\eta}(t,x)\Vert\leq C_\eta\Vert x\Vert$.

\end{enumerate}\end{proof}
\begin{rem}\label{rem:DC_redef_const}
If $\gamma(t,\mathbf{0})=\eta(t,\mathbf{0})=\Gamma~ \forall t\geq 0$ for some constant, known $\Gamma\in\mathbb{R}^n$ then $\tilde{\gamma}(t,x)=\gamma(t,x)-\Gamma$ and $\tilde{\eta}(t,x)=\eta(t,x)-\Gamma$.
\end{rem}
\subsection{Control Policy Synthesis}
We consolidate our results to construct a PWQ CLF $V(t,x(t))$ and stabilizing PWA policy $u^*(t,x(t))$ satisfying the properties in Lemma \ref{lem:CLF} for system \eqref{eq:DC} using the technique presented in section \ref{sec:lift} and the redefined dynamics \eqref{eq:DC_redef} from Lemma \ref{lem:DC_redef}. The main idea is to express the convex PWA functions $\tilde{\gamma}(\cdot)$, $\tilde{\eta}(\cdot)$ as the solution to a set of linear inequalities and quadratic equalities in a lifted space (as in \eqref{eq:max_solnset}) and then derive certificates to check if the candidate $V(\cdot)$ and $u^*(\cdot)$ satisfies the properties given in Lemma \ref{lem:CLF}. 

Consider CLF candidates for system \eqref{eq:DC_redef} (or equivalently, \eqref{eq:DC}) of the form 
\begin{align}\label{eq:clf_cand}
V(t,x)=\begin{bmatrix}x\\\tilde{\gamma}(t,x)\\\tilde{\eta}(t,x)
\end{bmatrix}^\top P \begin{bmatrix}x\\\tilde{\gamma}(t,x)\\\tilde{\eta}(t,x)
\end{bmatrix}.
\end{align}
Note that since $\tilde{\gamma}(t,\cdot), \tilde{\eta}(t,\cdot)$ are PWA in $x$, $V(t,\cdot)$ is PWQ in $x$. The polytopic partition and the quadratic functions describing $V(t,\cdot)$ are implicitly defined by $\tilde{\gamma}(t,\cdot)$ and $\tilde{\eta}(t,\cdot)$.

Now we proceed to construct the policy class for our stabilizing policy $u^*(t,x(t)$. Recursively define the following functions
\begin{equation*}
\begin{aligned}
    \tilde{\gamma}_{j}(t,x)=\max\{\tilde{\gamma}_{j-1}(t,x), \tilde{E}_j(t)x+\tilde{d}_j(t)\}~~j=2,\dots,\alpha\\
    \tilde{\eta}_{l}(t,x)=\max\{\tilde{\eta}_{l-1}(t,x),\tilde{H}_l(t)x+\tilde{f}_l(t)\}~~l=2,\dots,\beta
\end{aligned}
\end{equation*}
with $\tilde{\gamma}_1(t,x)=\tilde{E}_1(t)x+\tilde{d}_1(t)$ and $\tilde{\eta}_1(t,x)=\tilde{H}_1(t)x+\tilde{f}_1(t)$ and define
\begin{align}\label{eq:pol_basis}
\chi(t,x)=[1\ x^\top\ \tilde{\gamma}_1(t,x)^\top\ \dots\tilde{\gamma}_\alpha(t,x)^\top\ \tilde{\eta}_1(t,x)^\top\nonumber\\
\dots\tilde{\eta}_\beta(t,x)^\top]^\top.
\end{align} 
Due to the nested maximizations property as noted earlier, we have $\tilde{\gamma}(t,x)=\tilde{\gamma}_\alpha(t,x)$ and $\tilde{\eta}(t,x)=\tilde{\eta}_\beta(t,x)$.
Similar to \eqref{eq:max_solnset}, we have for any $t\geq 0$ that $(x,\tilde{\gamma}(t,x),\tilde\eta(t,x))$ is given as the solution to the following linear inequalities and quadratic equalities in the space $(x,\tilde\gamma_1, \dots, \tilde\gamma_\alpha, \tilde\eta_1,\dots,\tilde\eta_\beta)\in\mathbb{R}^{n(1+\alpha+\beta)}$ (where $x,\tilde\gamma_1, \dots, \tilde\eta_\beta$ are arbitrary).
\begin{equation}
\begin{aligned}\label{eq:dyn_solnset}
    &\tilde\gamma_1=\tilde E_1(t)x+\tilde d_1(t),\\
    &\tilde\gamma_{j}\geq \tilde\gamma_{j-1},\ \tilde\gamma_{j}\geq \tilde E_{j}(t)x+\tilde d_j(t),\\
    &(\tilde\gamma_{j}-\tilde\gamma_{j-1})^\top(\tilde\gamma_j-\tilde E_{j}(t)x-\tilde d_j(t))=0,\ j=2,\dots, \alpha,\\
    &\tilde\eta_1=\tilde H_1(t)x+\tilde f_1(t),\\
    &\tilde\eta_{l}\geq \tilde\eta_{l-1},\ \tilde\eta_{l}\geq \tilde E_{l}(t)x+\tilde f_l(t),\\
    &(\tilde\eta_{l}-\tilde\eta_{l-1})^\top(\tilde\eta_l-\tilde E_{l}(t)x-\tilde f_l(t))=0,\ l=2,\dots, \beta.
\end{aligned}
\end{equation}
For defining our policy class, we make the following assumption on the feedback structure of system \eqref{eq:DC}.
\begin{assumption}\label{assmp:fb_strct}
There exists a known matrix $C\in\mathbb{R}^{p\times (1+n(1+\alpha+\beta))}$ such that $C\chi(t,x)$ is observed. 
\end{assumption}
With this assumption, we consider policy candidates parametrized by $K\in\mathbb{R}^{m\times p}$ of the form 
\begin{align}\label{eq:pol_cand}
u(t,x)=KC\chi(t,x).
\end{align}
Since $\chi(t,\cdot)$ as defined in \eqref{eq:pol_basis} is PWA in $x$, the policy $u(t,\cdot)$ is PWA in $x$ where the time-varying polytopic partition $\sqcup_{i=1}^{n_r(t)}\tilde{\mathcal{X}}_i(t)=\mathcal{X}$ and affine functions $F_i(t)x+g_i$ in \eqref{eq:tvPWA_pol} are defined implicitly. 

The following theorem is our main result that describes LMIs whose solution yields a CLF $V(t,x(t))$ and stabilizing policy $u(t,x(t))$ of the form discussed above. Before stating the theorem, we require some definitions. Let
\begin{equation}
\begin{aligned}\label{eq:aux_dfs}
 x^+=Ax+\tilde{\gamma}_\alpha(t,x)-\tilde\eta_\beta(t,x)+BKC\chi(t,x)\\
 \chi^+(t,x)=[\chi(t,x)^\top \tilde\gamma_1(t+1,x^+)^\top\dots\tilde\gamma_\alpha(t+1,x^+)^\top\\ \tilde\eta_1(t+1,x^+)^\top\
 \dots\tilde\eta_\beta(t+1,x^+)^\top]^\top
\end{aligned}
\end{equation}
and define the following selection matrices:
\begin{itemize}
    \item $S$ such that $S\chi^+(t,x)=\chi(t,x)$,
    \item $S_\lambda$ such that $S_\lambda\chi^+(t,x)=[1\ x^\top\ \tilde\gamma(t,x)^\top\ \tilde\eta(t,x)^\top]^\top$,
    \item $S_x$ such that $S_x\chi^+(t,x)=[1\  x^\top]^\top$,
    \item $S^+$ such that $S^+\chi^+(t,x)=\chi(t+1,x^+)$,
    \item $G_{\gamma 1},G_{\eta 1}$ such that $G_{\gamma 1}\chi(t,x)=I_n$, $G_{\eta 1}\chi(t,x)=I_n$ respectively,
    \item $G_{\gamma j}, G_{\eta l}$ such that $G_{\gamma j}\chi(t,x)=\tilde\gamma_j(t,x)-\tilde\gamma_{j-1}(t,x)$, $G_{\eta l}\chi(t,x)=\tilde\eta_l(t,x)-\tilde\eta_{l-1}(t,x)$ for $j=2,\dots,\alpha$ and $l=2,\dots,\beta$ respectively,
    \item $G_{\gamma j}^k$, $G_{\eta l}^k$ such that $G^k_{\gamma j}\chi(t,x)=\tilde\gamma_j(t,x)-\tilde E_j^kx-\tilde d_j^k$, $G^k_{\eta l}\chi(t,x)=\tilde\eta_l(t,x)-\tilde H_l^kx-\tilde f_l^k$ for $j=1,\dots,\alpha$ and $l=1,\dots,\beta$ respectively.
\end{itemize}
Note that $S^+$ is affine in the unknown gain $K$ from \eqref{eq:aux_dfs}. All the other matrices are constant, known matrices that depend on the problem data.
\begin{theorem}\label{thm:pol_clf_synth}
Fix some $\rho_1\in\mathbb{R}^+, \rho_3\in(0,1)$. Define the ellipsoidal set  $\tilde{\mathcal{U}}=\{u\in\mathbb{R}^m:\ u^\top Q_uu\leq 1\}\subset \mathcal{U}$ where $Q_u^\top=Q_u$, $Q_u\succ 0$. Suppose that there exist real-valued matrices $T, V$, diagonal matrix $J$, non-negative matrices $M_{\gamma j}$, $M_{\eta l}$, $\bar M_{\gamma j}$, $\bar M_{\eta l}$, $\tilde M_{\gamma j}$, $\tilde M_{\eta l}$, $M^+_{\gamma j}$, $M^+_{\eta l}$, diagonal matrices $D_{\gamma l}$, $D_{\eta l}$, $\bar D_{\gamma j}$, $\bar D_{\eta l}$ $\tilde D_{\gamma l}$, $\tilde D_{\eta l}$, $D^+_{\gamma l}$, $D^+_{\eta l}$
for $j=1,\dots, \alpha$ and $l=1,\dots,\beta$, such that for some symmetric matrix $P^*$ and feedback gain $K^*$ we have:
    \small
\begin{align}\label{eq:lyap_pos_Lmi}
&\qquad~~~~~\text{(Positive-definite Lyapunov Function)}\nonumber\\
    &(S_\lambda S)^\top P^*S_\lambda S-S_x^\top\begin{bmatrix}0&\mathbf{0}\\\mathbf{0}&\rho_1 I_n\end{bmatrix}S_x \nonumber\\
    &-S^\top\text{He}(\sum_{j=1}^\alpha G^\top_{\gamma j}( D_{\gamma j}+M_{\gamma j})G_{\gamma j}^k+\sum_{l=1}^\beta G^\top_{\eta l}( D  _
    {\eta l}+M_{\eta l})G_{\eta l}^k)S \succeq 0\nonumber\\ &~~~~~k=1,\dots, n_V\\
    &~~~~~~~[P^*]_{11}=0\label{eq:Pzero}
\end{align}

\begin{align}\label{eq:input_constr_Lmi}
&\qquad~~~~~~~~~~~~~~\text{(Input Constraints)}\nonumber\\
    &(CS)^\top\begin{bmatrix}1&\mathbf{0}\\\mathbf{0}&-Q_u\end{bmatrix}CS\nonumber\\
    &-S^\top\text{He}(\sum_{j=1}^\alpha G^\top_{\gamma j}( \bar{D}_{\gamma j}+\bar{M}_{\gamma j})G_{\gamma j}^k+\sum_{l=1}^\beta G^\top_{\eta l}( \bar{D}_
    {\eta l}+\bar{M}_{\eta l})G_{\eta l}^k)S \succeq 0\nonumber\\ &~~~~~k=1,\dots, n_V
\end{align}

\begin{align}\label{eq:lyap_decr_LMI}
 &\text{(Decreasing Lyapunov Function)}\nonumber\\
&\left[\begin{array}{ccc}
   Z^{k}_1 & V &\multirow{2}{*}{$\mathbf{Z}^q_2$}\\
   \star & 2\text{He}(T)-Y &\\
   \multicolumn{2}{c}{\star} & Z_3 
\end{array}\right]\succeq 0 \nonumber\\
&~~~~ k=1,\dots, n_V, ~q=1,\dots, n_V
\end{align}
\normalsize
where
\small
\begin{align*}
    Z^{k}_1&=(1-\rho_3)S^\top P^* S\\
    &-S^\top\text{He}(\sum_{j=1}^\alpha G^\top_{\gamma j}( \tilde{D}_{\gamma j}+\tilde{M}_{\gamma j})G_{\gamma j}^k+\sum_{l=1}^\beta G^\top_{\eta l}( \tilde{D}_
    {\eta l}+\tilde{M}_{\eta l})G_{\eta l}^k)S,\\
    {Z}^q_2&=\begin{bmatrix}U^q&V\\\mathbf{0}&T^\top\end{bmatrix}, \quad Z_3=\begin{bmatrix}
    J&\mathbf{0}\\\star&2I-J
    \end{bmatrix},\end{align*}
    \begin{align*}
    U^q&=S^{+\top}[S^\top_\lambda\ G^\top_{\gamma 1}\ G_{\gamma 1}^{q\top}\dots\ G^\top_{\gamma \alpha}\ G_{\gamma \alpha}^{q\top}\ G^\top_{\eta 1}\ G_{\eta 1}^{q\top}\dots G^\top_{\eta \beta}\ G_{\eta \beta}^{q\top}],\\
    Y&=\text{blkdiag}(P^*,Y_{\gamma 1},\dots,Y_{\gamma\alpha},Y_{\eta 1},\dots,Y_{\eta\beta}),\\
    Y_{\gamma j}&=\begin{bmatrix}
    \mathbf{0}&D^+_{\gamma j}+M^+_{\gamma j}\\\star &\mathbf{0}
    \end{bmatrix}, Y_{\eta j}=\begin{bmatrix}
    \mathbf{0}&D^+_{\eta j}+M^+_{\eta j}\\\star &\mathbf{0}
    \end{bmatrix}.
\end{align*}
\normalsize

Then the origin of the closed-loop system \eqref{eq:DC} with $u^*(t,x)=K^*C\chi(t,x)$ is exponentially stable and $u^*(t,x)\in\mathcal{U}~\forall x\in\mathcal{X},t\geq0$.
\end{theorem}
\begin{proof}
Multiplying the $k$th inequality in \eqref{eq:lyap_pos_Lmi} by $\chi^+(t,x)$ on both sides for $x\in\mathcal{X}, t\geq 0$, gives
\begin{align*}
    V(t,x)-\rho_1\Vert x\Vert^2&\geq \sum_{j=1}^\alpha\chi^\top(t,x)G^\top_{\gamma j}( D_{\gamma j}+M_{\gamma j})G_{\gamma j}^k\chi(t,x)\\
    &+\sum_{l=1}^\beta\chi^\top(t,x)G^\top_{\eta l}( D_{\eta l}+M_{\eta l})G_{\eta l}^k\chi(t,x). 
\end{align*}
From Assumption \ref{assmp:unc}, we know that $\mathbf{W}(t)\in\mathcal{W}\Rightarrow \forall t\geq 0, \exists v(t)\in(\mathbb{R}^+)^{n_V}, \mathbf{1}^\top v(t)=1: \sum_{k=1}^{n_V}[v(t)]_kW^k=\mathbf{W}(t)$. Also notice that the matrices $G^k_{\gamma j},G^k_{\eta j}$ are linear in $W^k$. Let $\bar{G}_{\gamma j}(t)$, $\bar{G}_{\eta l}(t)$ be matrices such that $\tilde\gamma_j(t,x)-\tilde{E}_j(t)x-\tilde{d}_j(t)=\bar{G}_{\gamma j}(t)\chi(t,x)$,  $\tilde{\eta}_l(t,x)-\tilde{H}_l(t)x-\tilde{f}_l(t)=\bar{G}_{\eta l}(t)\chi(t,x)$ for $j=1,\dots,\alpha$ and $l=1,\dots,\beta$ respectively. Then we have $\bar{G}_{\gamma j}=\sum_{k=1}^{n_V}[v(t)]_kG^k_{\gamma j}$ and $\bar{G}_{\eta l}=\sum_{k=1}^{n_V}[v(t)]_kG^k_{\eta l}$. Thus multiplying the $k$th inequality above by $[v(t)]_k$ and adding all gives
\begin{align*}
    &V(t,x)-\rho_1\Vert x\Vert^2\geq\\
    &\sum_{j=1}^\alpha \chi(t,x)^\top G^\top_{\gamma j}(D_{\gamma j}+M_{\gamma j})\bar{G}_{\gamma j}(t)\chi(t,x)\\
    &+\sum_{l=1}^\beta  \chi(t,x)^\top G^\top_{\eta l}(D_{\eta l}+M_{\eta l})\bar{G}_{\eta l}(t)\chi(t,x)
\end{align*}
  From Proposition \ref{prop:lmbd_graph}, we have that the two terms on the right are positive and so
  $$\rho_1\Vert x\Vert^2\leq V(t,x) ~~\forall x\in\mathcal{X}, t\geq 0. $$
  Similarly, \eqref{eq:input_constr_Lmi} gives for any $x\in\mathcal{X}, t\geq 0$
  \begin{align*}
      &1-u^*(t,x)^\top Q_uu^*(t,x)\geq\\
      &\sum_{j=1}^\alpha \chi(t,x)^\top G^\top_{\gamma j}(D_{\gamma j}+M_{\gamma j})\bar{G}_{\gamma j}(t)\chi(t,x)\\
    &+\sum_{l=1}^\beta  \chi(t,x)^\top G^\top_{\eta l}(D_{\eta l}+M_{\eta l})\bar{G}_{\eta l}(t)\chi(t,x)
  \end{align*}
  for any $x\in\mathcal{X}, t\geq 0$. From Proposition \ref{prop:lmbd_graph}, we have that the two terms on the right are positive and so
  $$ u^*(t,x)\in\tilde{\mathcal{U}}\subset\mathcal{U}~\forall x\in\mathcal{X}, t\geq 0.$$
  
  From lemma~\ref{lem:DC_redef}, we have $\Vert\tilde\gamma(t,x)\Vert \leq C_\gamma\Vert x\Vert$, $
  \Vert\tilde\eta(t,x)\Vert\leq C_\eta\Vert x\Vert$ for some $C_\gamma, C_\eta\in\mathbb{R}^+$. Using \eqref{eq:Pzero} and setting $\rho_2=\lambda_{max}(P^*)(1+C_\gamma^2+C_\eta^2)$, we get $$V(t,x)\leq \rho_2\Vert x\Vert^2~~\forall x\in\mathcal{X}, t\geq 0. $$   
  Observe that $Z^{k}_1, Z^q_2$ in inequality \eqref{eq:lyap_decr_LMI} are linear in $G^k_{\gamma j}, G^k_{\eta l}$ and $G^q_{\gamma j}, G^q_{\eta l}$ respectively. Define $Z_1(t)$ by replacing $G^k_{\gamma j}, G^k_{\eta l}$ with $\bar{G}_{\gamma j}(t)$, $\bar{G}_{\eta l}(t)$ in the definition of $Z^{k}_1$ and see that $Z_1(t)=\sum_{k=1}^{n_V}[v(t)]_kZ^{k}_1$. Similarly, define $Z_2(t+1)$ by replacing $G^q_{\gamma j}, G^q_{\eta l}$ with $\bar{G}_{\gamma j}(t+1)$, $\bar{G}_{\eta l}(t+1)$ in the definition of $Z^q_2$ and see that  $Z_2(t+1)=\sum_{q=1}^{n_V}[v(t+1)]_qZ^q_2$. For a fixed $q$, multiplying the $k$th inequality in \eqref{eq:lyap_decr_LMI} by $[v(t)]_k$ and adding the $n_V$ inequalities gives 
  $$\left[\begin{array}{ccc}
   Z_1(t) & V &\multirow{2}{*}{$\mathbf{Z}^q_2$}\\
   \star & 2\text{He}(T)-Y &\\
   \multicolumn{2}{c}{\star} & Z_3 
\end{array}\right]\succeq0.$$
Again, multiplying the $q$th inequality above with $[v(t+1)]_q$ and adding the $n_V$ resulting inequalities gives
 $$\left[\begin{array}{ccc}
   Z_1(t) & V &\multirow{2}{*}{$Z_2(t+1)$}\\
   \star & 2\text{He}(T)-Y &\\
   \multicolumn{2}{c}{\star} & Z_3 
\end{array}\right]\succeq 0.$$
Now we proceed to eliminate the variables $J, T$ and $V$ as follows. Using the AM-GM inequality, we have $J^{-1}\succeq2I-J$ which we replace in $Z_3$. Then, we use Schur's complement and Young's inequality to eliminate $J$. Variables $T$ and $V$ are eliminated by employing the reciprocal-projection lemma (\cite[Section 2.5.4, 1]{caverly2019lmi}). Multiplying both sides of the resulting inequality with $\chi^+(t,x)$ for $x\in\mathcal{X}$, gives
\begin{align*}
    &-(V(t+1,x^+)-\rho_3V(t,x))\geq\\
    &\sum_{j=1}^\alpha \chi(t,x)^\top G^\top_{\gamma j}(\tilde D_{\gamma j}+\tilde{M}_{\gamma j})\bar{G}_{\gamma j}(t)\chi(t,x)+\\
    &\sum_{l=1}^\beta  \chi(t,x)^\top G^\top_{\eta l}(\tilde D_{\eta l}+\tilde M_{\eta l})\bar{G}_{\eta l}(t)\chi(t,x)+\\
    &\sum_{j=1}^\alpha \chi(t+1,x^+)^\top G^\top_{\gamma j}(D^+_{\gamma j}+M^+_{\gamma j})\bar{G}_{\gamma j}(t+1)\chi(t+1,x^+)+\\
    &\sum_{l=1}^\beta  \chi(t+1,x^+)^\top G^\top_{\eta l}(D^+_{\eta l}+M^+_{\eta l})\bar{G}_{\eta l}(t+1)\chi(t+1,x^+).
\end{align*}
From Proposition \ref{prop:lmbd_graph}, we have that the four terms on the right are positive and so
  $$V(t+1,x^+)\leq \rho_3 V(t,x) ~~\forall x\in\mathcal{X}, t\geq 0.$$
  Thus by Lemma \ref{lem:CLF}, the origin of the closed-loop system \eqref{eq:DC} with $u^*(t,x)=K^*C\chi(t,x)$ is exponentially stable.
  \end{proof}
\begin{rem}
Note that although \eqref{eq:lyap_pos_Lmi},\eqref{eq:Pzero} make $V(t,x)$ positive-definite on $\mathbb{R}^n$, the matrix $P^*\in\mathbb{R}^{(1+3n)\times(1+3n)}$ can be indefinite.
\end{rem}
\begin{rem}
The LMI \eqref{eq:lyap_decr_LMI} can be conservative in practice. Alternatively, it can be replaced by the following BMI.
\begin{align}\label{eq:lyap_decr_bmi}
 &\text{(Decreasing Lyapunov Function)}\nonumber\\
&\begin{bmatrix}
   \mathbf{Z}^k_1 & \mathbf{Z}^q_2\\\star & \mathbf{Z}_3 
\end{bmatrix}\succeq 0 \nonumber\\
&~~~~ k=1,\dots, n_V, ~q=1,\dots, n_V
\end{align}
\normalsize
where
\small
\begin{align*}
    \mathbf{Z}^k_1&=(1-\rho_3)S^\top P^* S+\tilde{r}S_x^\top\begin{bmatrix}-1&\mathbf{0}\\\mathbf{0}& Q_x\end{bmatrix}S_x-UV^\top-VU^\top\\
    &-S^\top\text{He}(\sum_{j=1}^\alpha G^\top_{\gamma j}( \tilde{D}_{\gamma j}+\tilde{M}_{\gamma j})G_{\gamma j}^k+\sum_{l=1}^\beta G^\top_{\eta l}( \tilde{D}_
    {\eta l}+\tilde{M}_{\eta l})G_{\eta l}^k)S,\\
    \mathbf{Z}^q_2&=V-U^qT, \quad \mathbf{Z}_3= T+T^\top-Y,\end{align*}
    \begin{align*}
    U^q&=S^{+\top}[S^\top_\lambda\ G^\top_{\gamma 1}\ G_{\gamma 1}^{q\top}\dots\ G^\top_{\gamma \alpha}\ G_{\gamma \alpha}^{q\top}\ G^\top_{\eta 1}\ G_{\eta 1}^{q\top}\dots G^\top_{\eta \beta}\ G_{\eta \beta}^{q\top}],\\
    Y&=\text{blkdiag}(P^*,Y_{\gamma 1},\dots,Y_{\gamma\alpha},Y_{\eta 1},\dots,Y_{\eta\beta}),\\
    Y_{\gamma j}&=\begin{bmatrix}
    \mathbf{0}&D^+_{\gamma j}+M^+_{\gamma j}\\\star &\mathbf{0}
    \end{bmatrix}, Y_{\eta j}=\begin{bmatrix}
    \mathbf{0}&D^+_{\eta j}+M^+_{\eta j}\\\star &\mathbf{0}
    \end{bmatrix}.
\end{align*}
\normalsize 
\end{rem}
\section{Numerical Examples}\label{sec:numerics}
In this section we illustrate the proposed control synthesis procedure on two examples: A) an inverted pendulum against a soft wall and B) the human-robot collaboration scenario from section~\ref{sec:PD}. For each example, we obtain the DC decomposition, verify our assumptions and synthesize feedback policies which are demonstrated via simulations.

\subsection{Inverted pendulum against Wall}\label{ssec:IPSW}
\begin{figure}[h]
    \centering
    \includegraphics[width=0.45\columnwidth]{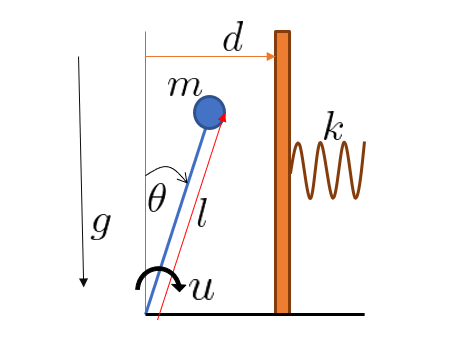}
    \caption{Inverted pendulum with soft wall.}
    \label{fig:IPSW}
\end{figure}
\subsubsection{System Model}
Consider the inverted pendulum of mass $m$ and length $l$ in Figure~\ref{fig:IPSW}. The state of the system  at time $t$, $x(t)=[\theta(t)\ \dot\theta(t)]^\top$, is described by the deviation $\theta$ of the pendulum from the vertical, and its angular velocity $\dot\theta$. At distance $d(t)$ from the vertical, there is a soft wall modeled as a spring with spring constant $k(t)$. Both the spring constant $k(t)$ and distance $d(t)$ are uncertain and time-varying. The torques acting on the pendulum are $mgl\sin\theta$ (due to gravity), $u(t)$ (control torque) and the moment due to the spring given by $\lambda(t,x)=l\cos\theta\max\{0,k(t)(l\sin\theta-d(t))\}$ in the counter-clockwise direction. Assuming small deviations ($\sin\theta\approx\theta$) and using forward-Euler discretization, we can write the system dynamics as the PWA model
\begin{equation}\label{eq:IPSW}
\begin{aligned}
    x(t+1)=&\underbrace{\begin{bmatrix}
    1&dt\\dt\frac{g}{l}&1
    \end{bmatrix}}_{A}x(t)+\underbrace{\begin{bmatrix}0\\\frac{dt}{ml^2}\end{bmatrix}}_{B}u(t)\\
    &~~~-\underbrace{\begin{bmatrix}
    0\\\frac{dt}{ml}
    \end{bmatrix}\max\{0,k(t)l\theta(t)-k(t)d(t)\}}_{\eta(t,x(t))}.
\end{aligned}
\end{equation}
We assume that the uncertain parameters $\mathbf{W}(t)=[k(t)\ k(t)d(t)]$ lie in some convex polytope $\mathcal{W}=\text{conv}\{[k^i\ k_d^i]\}_{i=1}^{n_V}$ as in Assumption~\ref{assmp:unc}. We can also assume $k^i,k_d^i\geq0 ~\forall i=1,\dots, n_V$ to represent non-negative spring constants and distances. Thus $k(t), k(t)d(t)\geq 0$ and so we have $\eta(t,\mathbf{0})=\max\{0,-k(t)d(t)\}\begin{bmatrix}
    0&\frac{dt}{ml}
    \end{bmatrix}^\top=\mathbf{0}=\gamma(t,\mathbf{0})$ in Assumption~\ref{assmp:eqlbrm}. From Remark \ref{rem:DC_redef_const}, the redefined DC decomposition is simply $\tilde{\eta}(t,x)=\eta(t,x)$. 

For policy synthesis, we construct $\chi(t,x)=[1\ x^\top\ \eta_1(t,x)^\top \eta_2(t,x)^\top]^\top$ where $\eta_1(t,x)=\mathbf{0}$ and $\eta_2(t,x)=\max\{\eta_1(t,x),[0\  \frac{dtk(t)}{ml}(l\theta-d(t))]^\top\}$. For the observation model, we assume that the full state $x(t)$ is observable and that we can also observe the contact torque $\lambda(t,x)$. Thus, the matrix $C\in\mathbb{R}^{4\times7}$ in Assumption \ref{assmp:fb_strct} is given by
$C=
\begin{bmatrix}
[I_7]^\top_1&[I_7]^\top_2&[I_7]^\top_3&[I_7]^\top_7
\end{bmatrix}^\top.
$
We now proceed to synthesize a control policy using theorem \ref{thm:pol_clf_synth} for a specific realization of this system.
\subsubsection{Simulations}
The various parameters in \eqref{eq:IPSW} are given as follows (in SI units): $dt=0.01$, $m=1 $, $l=2$, $g=9.8$, $\mathcal{W}=\text{conv}\{[10\ 2], [10\ 2.5],[5\ 1], [5\ 1.25]\}$. The state and input constraints are given by $\mathcal{X}=[-0.5,0.5]\times[-2,1]$, $\mathcal{U}=[-200,200]$. 

We use YALMIP with SDPT3 to find a feasible solution to the inequalities of theorem \ref{thm:pol_clf_synth}. The synthesized policy is tested on three different initial conditions $x_0=\{[0.18\ 0.8]^\top , [0.1\ 1]^\top, [0.22\ 0]^\top\}$ 
\begin{figure}[h]
    \centering
    \includegraphics[width=\columnwidth]{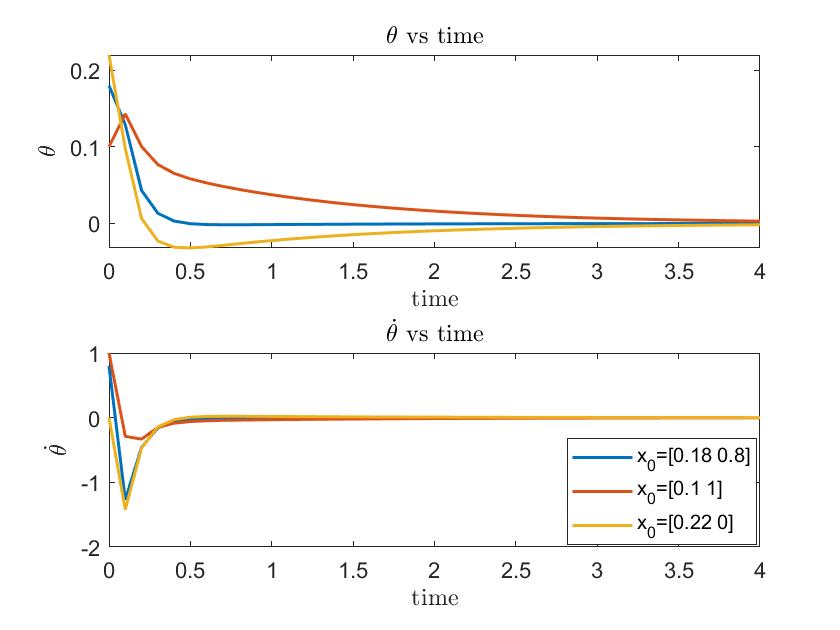}
    \caption{Deviation of the pendulum from the vertical and angular velocity $\theta(t)$, $\dot\theta(t)$ vs. time $t$ for three different initial conditions for system \eqref{eq:IPSW} in closed-loop with the policy $u^*(t,x)$ derived from theorem \ref{thm:pol_clf_synth}.}
    \label{fig:IPSW_states}
\end{figure}
\begin{figure}[h]
    \centering
    \includegraphics[width=\columnwidth]{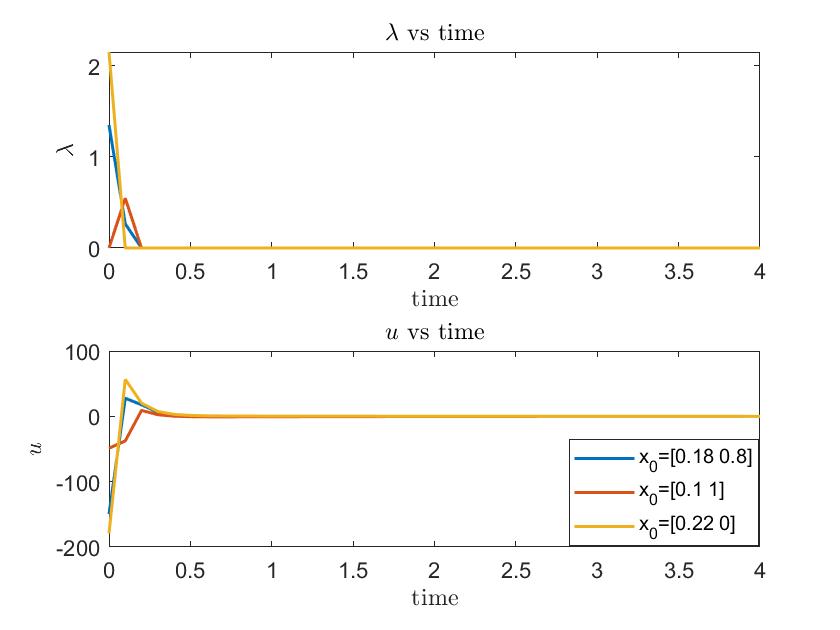}
    \caption{Contact torque $\lambda(t,x(t))$ and control input $u^*(t,x(t))$ vs. time $t$ for three different initial conditions for system \eqref{eq:IPSW} in closed-loop with the policy $u^*(t,x)$ derived from theorem \ref{thm:pol_clf_synth}.}
    \label{fig:IPSW_ulmbd}
\end{figure}
\subsubsection{Discussion}
 Figures~\ref{fig:IPSW_states},\ref{fig:IPSW_ulmbd} show that the system trajectories converge to the origin for all three initial conditions while satisfying state and input constraints. The uncertainty description tells us that $d(t)\in[0.2,0.25]$, or $\dfrac{d(t)}{l}\in[0.1,0.125]$. So the pendulum starts close to the wall in the considered initial conditions, which is confirmed by the first plot in Figure~\ref{fig:IPSW_ulmbd}. The first and last trajectories start with the pendulum in contact with the wall whereas in the second trajectory, the pendulum bounces off the wall before converging to the origin. The control signal in the second plot of Figure~\ref{fig:IPSW_ulmbd} has kinks when contact is made or broken.
\subsection{Human-Robot Collaboration}
\subsubsection{System Model} 
\begin{figure}[h]
    \centering
    \includegraphics[width=0.7\columnwidth]{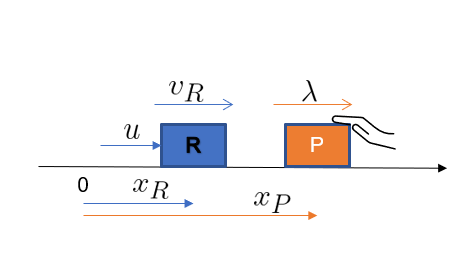}
    \caption{Robot cooperating with human to transport a payload.}
    \label{fig:HR}
\end{figure}
Suppose we have a collaborative setup where a robot is transporting a payload with assistance from a human as illustrated in figure~\ref{fig:HR}. Denote the position and velocity of the robot by $x_R$ and $v_R$ respectively. The velocity $v_R$ is controlled by $u$ via first order discrete-time dynamics. Denote the position of the payload as $x_P$ and let $x(t)=[x_R(t)\ v_R(t)\ x_P(t)]^\top$ describe the state of the system at time $t$. The velocity of the payload is determined by the human's policy as $\lambda(t,x(t))$ given by
\begin{align*}
    \lambda(t,x(t))=&\begin{cases} 
    K(t)(x_R(t)-x_P(t)) & c(t)x_P(t)>x_R(t)\\
    K(t)(c(t)-1)x_P(t)) & c(t)x_P(t)\leq x_R(t)\\
    \end{cases}
\end{align*}
where $c(t)\in(0,1], K(t)\geq 0$ are uncertain parameters of the human's policy. The contact model between the robot and the payload is given by that of inelastic collision. The model of the system is thus given by
\small
\begin{align}\label{eq:HR_eg}
    \begin{bmatrix}x_R(t+1)\\v_R(t+1)\\x_P(t+1)\end{bmatrix}=\begin{bmatrix}
    x_R(t)+dt\cdot v_R(t)\\
    v_R(t)+dt\cdot u(t)\\
    \max\{x_R(t)+dt\cdot v_R(t), x_P(t)+dt\cdot \lambda(t,x(t))\}
    \end{bmatrix}
\end{align}
\normalsize
where $dt$ is the discretization time-step. Observe that because $K(t), c(t)$ are uncertain and time-varying, the polytopic partition of the PWA system \eqref{eq:HR_eg} is uncertain and time-varying too. Also notice that the set $\mathcal{X}_I=\{x\in\mathbb{R}^3: x_R\leq x_P\}$ is invariant for the dynamics \eqref{eq:HR_eg} and we assume that $x(0)\in\mathcal{X}_I$. 
When $x_p(t)\geq0$, the human drives the payload towards the origin if it is close to the robot ($c(t)x_P(t)\leq x_R(t)$) and towards the robot otherwise. When $x_P(t)<0$, the payload is always pushed to the left (towards $-\infty$). 
Now we derive the DC decomposition of the system  model \eqref{eq:HR_eg}. The velocity that the human imparts to the payload, $\lambda(t,x)$, can be written as
\begin{align*}
 \lambda(t,x)=&-\max\{-K(t)(x_R-x_P),-K(t)(c(t)-1)x_P\}.
\end{align*}
Using this expression and the fact that  $\max\{a,b+c\}=\max\{a-b,c\}+b$, the system \eqref{eq:HR_eg} can be written in the form \eqref{eq:DC} as
\begin{equation}\label{eq:HR_DC}
\begin{aligned}
    x(t+1)=&Ax(t)+Bu(t)+\gamma(t,x(t))-\eta(t,x(t))
\end{aligned}
\end{equation}
where $[\gamma(t,x)]_i=[\eta(t,x)]_i=0$ for $i=1,2$ and $[\gamma(t,x)]_3=\max\{x_R+dt(v_R-\lambda(t,x)),x_P\}=\max\{x_R+dt(v_R-K(t)(x_R-x_P)),x_R+dt(v_R-K(t)(c(t)-1)x_P),x_P\}$, $[\eta(t,x)]_3=-dt\lambda(t,x)$. 

We assume that the uncertain parameters $\mathbf{W}(t)=[K(t)\ K(t)c(t)]$ lie in some convex polytope $\mathcal{W}=\text{conv}\{[K^i\ K_C^i]\}_{i=1}^{n_V}$ as in Assumption~\ref{assmp:unc}. We also assume that $K^i, K_c^i\geq 0, K_c^i<K^i~\forall i=1,\dots, n_V$ and so $K(t)\geq 0, c(t)\in[0,1)$. For any $\mathbf{W}(t)\in\mathcal{W}$, we have $\gamma(t,\mathbf{0})=\eta(t,\mathbf{0})=\mathbf{0}$. So from Remark \ref{rem:DC_redef_const}, the redefined DC decomposition is simply $\tilde{\gamma}(t,x)=\gamma(t,x)$, $\tilde{\eta}(t,x)=\eta(t,x)$. 

For policy synthesis, we construct $\chi(t,x)=[1\ x^\top\ \gamma_1(t,x)^\top\ \gamma_2(t,x)^\top\ \gamma_3(t,x)^\top\eta_1(t,x)^\top \eta_2(t,x)^\top]^\top$ where $\gamma_1(t,x)=[0\ 0\ x_R+dt(v_R-K(t)(x_R-x_P))]^\top$, $\gamma_2(t,x)=\max\{\gamma_1(t,x),[0\ 0\ x_R+dt(v_R-K(t)(c(t)-1)x_P)]^\top\}$, $\gamma_3(t,x)=\max\{\gamma_2(t,x),[0\ 0\ x_P]^\top\}$,  $\eta_1(t,x)=[0\ 0\ -dtK(t)(x_R-x_P)]^\top$ and $\eta_2(t,x)=\max\{\eta_1(t,x),[0\ 0\ -dtK(t)(c(t)-1)x_P]^\top\}$. For the observation model, we assume that the full state $x(t)$ is observable and that we can also observe the velocity imparted by the human, $\lambda(t,x)$. Thus, the matrix $C\in\mathbb{R}^{6\times19}$ in Assumption \ref{assmp:fb_strct} is given by
$C=
\begin{bmatrix}
[I_{19}]^\top_1&[I_{19}]^\top_2&[I_{19}]^\top_3&[I_{19}]^\top_4&[I_{19}]^\top_{13}&[I_{19}]^\top_{19}
\end{bmatrix}^\top
$ where the last two rows correspond to $[\gamma_3(t,x)]_3$ and $[\eta_3(t,x)]_3$.
We now proceed to synthesize a control policy using theorem \ref{thm:pol_clf_synth} for a specific realization of this system.
\subsubsection{Simulations}
\begin{figure*}
    \centering
    \includegraphics[width=\textwidth]{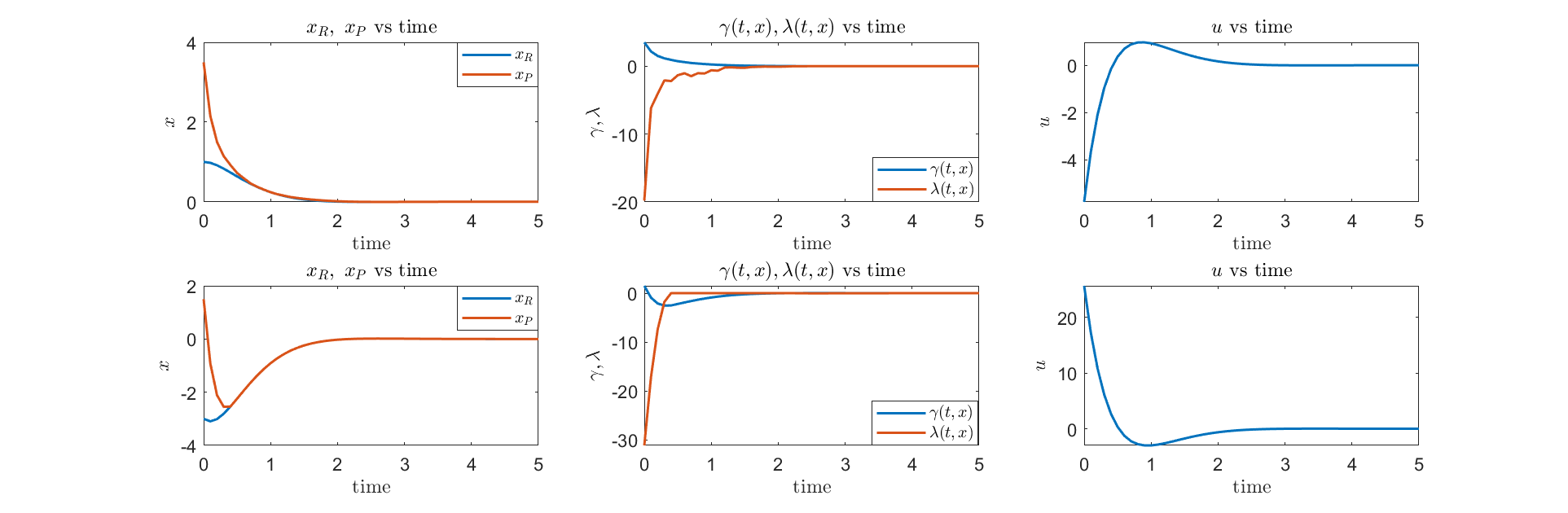}
    \caption{Closed-loop system behaviour of \eqref{eq:HR_eg} for under policy $u^*(t,x)$ derived from theorem \ref{thm:pol_clf_synth}. The plots on the first row correspond to the initial condition $x_0=[1\ 0\ 3.5]^\top$ and the plots on the second row correspond to the initial condition $x_0=[-3\ -2\ 1.5]^\top$.}
    \label{fig:HR_plots}
\end{figure*}
The various parameters in \eqref{eq:HR_DC} are given as follows (in SI units): $dt=0.01$, $\mathcal{W}=\text{conv}\{[10\ 6], [5\ 3],[10\ 8], [5\ 4]\}$.

We use YALMIP with SDPT3 to find a feasible solution to the inequalities of Theorem \ref{thm:pol_clf_synth}. The synthesized policy is tested on two different initial conditions $x_0=\{[1\ 0\ 3.5]^\top , [-3\ -2\ 1.5]^\top\}$ in $\mathcal{X}_I$. The two rows of plots in Figure~\ref{fig:HR_plots} correspond to simulations with the first and second initial conditions respectively. 

\subsubsection{Discussion}
 Figure~\ref{fig:HR_plots} shows that the system trajectories converge to the origin for both initial conditions. Since $K(t)\geq 0$, $c(t)\in[0,1)$ and $x(t)\in\mathcal{X}_I~\forall t\geq 0$, we see that the velocity imparted by the human $\lambda(t,x(t))$ is always negative. In the first row of plots for initial condition $x_0=[1\ 0\ 3.5]^\top$, we see that the robot is driven to the origin first while the payload follows behind. For the second initial condition, the robot starts at $x_R=-3$. The payload is pushed until it makes contact with the robot and then the robot pushes it to the origin with no further action from the human ($\lambda(t,x(t))=0$). 
 \section{Conclusion}
In this work, we analysed and designed feedback controllers for discrete-time Piecewise-Affine (PWA) systems with uncertainty in both the affine dynamics and the polytopic partition. Using the Difference-of-Convex (DC) decomposition of continuous PWA systems and the lifting technique discussed in section~\ref{sec:lift}, we provided LMI tests to synthesize quadratic Lyapunov functions and stabilizing affine policies in a higher dimensional space. When projected back to the state space, these resulted in time-varying PWQ Lyapunov functions and time-varying PWA feedback policies. We demonstrated our technique on two examples- an inverted pendulum against a soft wall and a human-robot collaboration task to transport a payload.
\bibliographystyle{ieeetr}
\bibliography{root.bib}

\end{document}